\documentclass{amsart}
\usepackage{amsfonts, amsmath}
\usepackage{amssymb,amsmath}
\usepackage{amsfonts,mathrsfs}
\usepackage{epstopdf}
\usepackage{graphicx}
\theoremstyle{plain}
\newtheorem{theorem}{Theorem}[section]

\theoremstyle{plain}

\theoremstyle{plain}

\newtheorem{lemma}[theorem]{Lemma}

\newcommand{\be}{\begin{equation}}
\newcommand{\ee}{\end{equation}}
\newcommand{\bea}{\begin{eqnarray}}
\newcommand{\eea}{\end{eqnarray}}
\newcommand{\bes}{\begin{subequations}}
\newcommand{\ees}{\end{subequations}}

\begin{document}
\author{Shohreh Gholizadeh Siahmazgi}
\email{ghols18@wfu.edu}

\author{Stephen B. Robinson}
\email{sbr@wfu.edu}

\address{Department of Mathematics, Wake Forest University, Winston-Salem, North Carolina 27109, USA}

\title{\large \bf   On The Convergence of the Variational Iteration Method Applied to Variable Coefficient Klein-Gordon Problems}

\begin{abstract}

In this paper, we give a  formulation of  the variational iteration method that makes it suitable  for the analysis of the solutions of Klein-Gordon equations with variable coefficients.
 We particularly study a Klein-Gordon problem which has solutions in terms of Airy functions. We prove that the sequence of approximate solutions generated by the variational iteration method for such Klein-Gordon equation converges to Airy functions.

  \end{abstract}

\maketitle

\section{Introduction}

In this paper we study several aspects of the Variational Iteration Method (VIM) within the context of Klein-Gordon problems with variable coefficients. Consider
\begin{equation}
    \begin{cases}
    \psi_{rr}-\psi_{tt}+V(r)\psi=0\;\;\text{in} \;\;\; [0,\infty)\times \mathbb{R},\\
     \psi(0,t)=g(t),\\
      \psi_r(0,t)=h(t).
    \end{cases}\label{pde-with-V(r)}
  \end{equation}
We call ~\eqref{pde-with-V(r)} an initial value problem in the sense that the space coordinate $r$ plays the role of the time coordinate $t$ in a more standard initial value problem. Problems of this sort arise in the study of quantum fields in curved spacetimes \cite{Birrell-Davies, mythesis}.

In the following sections we reformulate the VIM for problem ~\eqref{pde-with-V(r)} leading to the iteration equation
\begin{equation}
\psi_{n+1}(r,t)=\psi_n(r,t)+\int_{0}^{r}\lambda(r,s)\left(\psi_{n,ss}(s,t)-\psi_{n,tt}(s,t)+V(s)\psi_n (s,t)\right )ds,
\label{IterationEquation}
\end{equation}
where $\lambda$ is an appropriate {\em Lagrange multiplier}. We then study this iteration procedure for the particular case where $V(r)=r$. It turns out that
\[
\lambda(r,s)=\sum_{k=0}^{\infty}\alpha_k (s-r)^k,
\]
is an Airy function whose coefficients satisfy a standard recursion relation. We pose the question as to whether the VIM produces a sequence that converges uniformly to the known solution on closed and bounded intervals. Moreover we show that the series converges when we use either of the first two partial sums of $\lambda$ rather than the full series . We provide numerical evidence that the convergence is faster when the larger partial sum is used, as would be expected.

The literature documenting the effectiveness of VIM on a wide variety of examples has been growing rapidly over the last twenty years. Methods of this type actually appeared at least as far back as \cite{boulware}. In articles such as \cite{hej} and \cite{hejj} the VIM was introduced in its current form. Its application to wave equations and Klein-Gordon equations has been explored in articles such as \cite{hejj}, \cite{hartle-hawking}, \cite{tatari}, and \cite{kasumo}. Most of the literature presents interesting and compelling examples. Convergence proofs are rare and often simply assume that the iteration method is associated with a contraction mapping. 

In sections \ref{sec:vim}, the formulation for the VIM is discussed and an iteration formula for the approximate solutions to particular type of initial value problems is presented. In section  \ref{airyequation}, the VIM formulation is applied to a particular initial value problem which has solutions in terms of Airy function. The approximate solutions of such problem produced by the VIM is discussed. In section \ref{airyconvergence}, we prove the convergence of the sequence of the approximate solutions to the exact solution of the problem.

\section{Variational Iteration Method}
\label{sec:vim}

In this section we provide a formal derivation of the variational iteration method (VIM) for the problem ~\eqref{pde-with-V(r)}. The first steps follow a familiar idea of creating an iteration scheme from an initial value problem, but our approach also differs from the standard used in much of the VIM literature. We then introduce $\lambda$ which is usually called the general Lagrange multiplier. Once $\lambda$ is determined the iteration formula ~\eqref{IterationEquation} can be defined precisely.

\subsection{One step of the iteration}

Consider the initial value problem ~\eqref{pde-with-V(r)}. Define
$L\psi:=\psi_{rr}+V(r)\psi$. Using this notation, the differential equation in ~\eqref{pde-with-V(r)} takes the form
\bea
L\psi(r,t)-\psi_{tt}(r,t)=0\label{pde-with-V(r)-2nd-version}.
\eea
To set up one iteration step we choose a function (an initial guess) $\psi_0(r,t)$ satisfying the auxiliary conditions, i.e., $\psi_0(0,t)=g(t)$ and $\psi_t(0,t)=h(t)$ and look for a correction function $\phi$ such that $\psi_1=\psi_0+\phi$ satisfies
\begin{equation}
    \begin{cases}
    L\psi_1(r,t)-\psi_{0tt}(r,t)=0\;\;\text{in} \;\;\; [0,\infty)\times \mathbb{R},\\
     \psi_1(0,t)=g(t),\\
      \psi_{1r}(0,t)=h(t).
    \end{cases}\label{pde-with-V(r)-3nd-version}
  \end{equation}
It follows that finding $\psi_1(r,t)$ can be reduced to finding $\phi(r,t)$ such that
\bea
L\phi(r,t)+L\psi_0(r,t)-\psi_{0tt}(r,t)=0\label{phi}
\eea
with $\phi(0,t)=0$ and $\phi_r(0,t)=0$.

Let $L_0^{-1}$ represent the solution operator for
\begin{equation}
    \begin{cases}
    Lu(r,t)+f(r,t)=0\;\;\text{in} \;\;\; [0,\infty)\times \mathbb{R},\\
     u(0,t)=0,\\
      u_r(0,t)=0.
    \end{cases}\label{solution-operator}
  \end{equation} 
where $f$ is assumed to be a continuous function. In other words, $u=-L^{-1}_0f$ satisfies ~\eqref{solution-operator}, and $L^{-1}_0$ is the inverse of $L$ subject to the homogeneous auxiliary conditions.  Substituting $u=\phi$ and $f=L\psi_0(r,t)-\psi_{0tt}(r,t)$ in ~\eqref{solution-operator} we get
\bea
\phi=-L^{-1}_0 (L\psi_0(r,t)-\psi_{0tt}(r,t)).\label{phi2}
\eea
adding this correction function to the initial guess provides the first new iterate
\[
\psi_1=\psi_0+\phi=\psi_0-L^{-1}_0 (L\psi_0(r,t)-\psi_{0tt}(r,t)).
\]

\subsection{\textbf{Introducing $\lambda$ and a formula for $L^{-1}_0 $}}

Consider
\bea
L^{-1}_0 f=-\int_0^r \lambda(s,r)f(s)ds.\label{proposedlambda}
\eea
where $\lambda$ is to be determined. Supposing that $w(r,t)$ is any test function\footnote{This is often called "variation" in the context of the variational iteration method.} satisfying the homogeneous auxiliary conditions, i.e., $w(0,t)=0$ and $w_r(0,t)=0$. This means $w=L^{-1}_0 L w$. Hence
\bea
w&=&-\int_0^r \lambda(s,r)Lw(s,t)ds\nonumber \\
&=& -\int_0^r \lambda(s,r)w_{ss}(s,t)ds -\int_0^r \lambda(s,r)V(s)w(s,t)ds.\label{w1}
\eea
One can take the first integral after the second equality in  ~\eqref{w1} and apply integration by parts on it twice to find
\bea
-\int_0^r \lambda(s,r)w_{ss}(s,t)ds&=& \Big[\lambda(s,r)w_s(s,t)\Big]\Big|_{s=0}^{s=r}-\Big[\lambda_s(s,r)w(s,t)\Big]\Big|_{s=0}^{s=r}\nonumber \\ &&
+\int_0^r \lambda_{ss}(s,r)w(s,t)ds\nonumber \\
&=&\lambda(r,r)w_r(r,t)-\lambda_r(r,r)w(r,t)\nonumber \\&&
+\int_0^r \lambda_{ss}(s,r)w(s,t)ds.
.\label{w2}
\eea
Substituting ~\eqref{w2} into ~\eqref{w1} gives
\bea
0&=&-\lambda(r,r)w_r(r,t)+(\lambda_s(r,r)-1)w(r,t)\nonumber \\&&
-\int_0^r \Big(\lambda_{ss}(s,r)+V(s)\lambda(s,r)\Big)w(s,t)ds\label{w-integral}
\eea

Using judicious choices of test functions we derive the following initial value problem for $\lambda(s,r)$
\bes \bea
\lambda_{ss}(s,r)+V(s)\lambda(s,r)&=&0,\\
\lambda_s(r,r)&=&1,\\
\lambda (r,r)&=&0
\eea \label{lambda1} \ees

Solving for $\lambda$ is now standard. In much of the literature we have $V\equiv 1$ and so $\lambda(s,r)=s-r$. If $V(r)$ is analytic, then $\lambda(s,r)$ can be found by standard series solution methods, but is now more difficult to compute with.
\subsection{The VIM Formula}
The method discussed in the previous section can be applied to differential equations of the form
\begin{equation}
L[u]+N[u]=0,\label{LNg}
\end{equation}
where $L$ is some linear differential operator, $N[u]$ is possibly a non-linear differential operator. In the case discussed in the previous section $L=\partial_{rr}+V(r)$ and $N=-\partial_{tt}$. Given ~\eqref{phi} and ~\eqref{proposedlambda},  we find a formula for $\phi=\psi_1-\psi_0$
\bea
\phi=\psi_1-\psi_0=\int_0^r\lambda(s,r) (\partial_{ss}\psi_0+V(s)\psi_0-\partial_{tt}\psi_0)ds.
\eea
In general, one can write
\bea
\psi_{n+1}=\psi_n+\int_0^r\lambda(s,r) (\partial_{ss}\psi_n+V(s)\psi_n-\partial_{tt}\psi_n)ds,\label{iteration-new}
\eea
where $\psi_n$ and $\psi_{n+1}$ are the $n$th and $n+1$th iterations starting from $\psi_0$ respectively. For a differential equation of the general form $Lu+Nu=0$, the iteration scheme then will be
\bea
\psi_{n+1}=\psi_n+\int_0^r\lambda(s,r) (L\psi_n+N\psi_n)ds.
\eea
In the next chapters, we show that the iteration scheme proposed in ~\eqref{iteration-new} generates a sequence $\psi_n$ that converges to the solution of the desired differential equation in the limit $n\to \infty$.

\section{Airy Equation}
\label{airyequation}
We consider a particular type of initial value problem of form~\eqref{pde-with-V(r)} in which the potential term is linear; i.e., $V(r)=r$
\begin{equation}
    \begin{cases}
    \psi_{rr}(r,t)-\psi_{tt}(r,t)+r\psi(r,t)=0\;\;\;\;\text{in} \;[0,\infty)\times\mathbb{R},\\
     \psi(0,t)=e^{it},\\
       \psi_r(0,t)=0.
    \end{cases} \label{Airy3n}
  \end{equation}
The the solutions to ~\eqref{Airy3n} can be expressed in terms of Airy functions ~\cite{nagle}. Here, we first find the solutions to ~\eqref{Airy3n} using the standard method of solving partial differential equations and then apply the variational iteration method to this problem. Our main goal  is to test the convergence of the solutions constructed by the VIM method to the exact solutions of ~\eqref{Airy3n}. We first assume that ~\eqref{Airy3n} has a power series solution of form
\begin{equation}
\psi(r,t)= e^{it}\underset{k=0}{\overset{\infty}{\sum}}\alpha_n r^n.\label{Airy2-series}
\end{equation}
By substitution of ~\eqref{Airy2-series} into ~\eqref{Airy3n} and imposing the initial conditions, we find the following recursion relation for $\alpha_n$
\begin{equation}
\alpha_{n+2}=-\frac{\alpha_n+\alpha_{n-1}}{(n+1)(n+2)},\label{Airy2-recursion-exact}
\end{equation}
with $\alpha_0=1$, $\alpha_1=0$, and $\alpha_2=-\frac{1}{2}$. Note that ~\eqref{Airy2-recursion-exact} is the well-known recursion relation for the coefficients of the power series of an Airy function \cite{nagle}.\\
\subsection{VIM Applied to Airy Equation}
In order to apply the variational iteration method to ~\eqref{Airy3n}, we first have to determine $\lambda$. Given ~\eqref{lambda1}, we find the following system of differential equations for $\lambda(s,r)$,
\bes \bea
\lambda_{ss}(s,r)+s\lambda(s,r)&=&0,\\
\lambda_s(r,r)&=&1,\\
\lambda (r,r)&=&0
\eea \label{lambdastar} \ees
that are true for all $r, s \in [0,\infty)$. We let
\bea
\lambda(s,r)=\underset{n=0}{\overset{\infty}{\sum}}a_n(s-r)^n\label{lambda3}
\eea
with $a_0=0$, $a_1=1$, and $a_2=0$. By substitution of ~\eqref{lambda3} into ~\eqref{Airy3n}, we derive the following recurrence relation for $a_n$ for $n\geq 2$
\bea
a_{n+2}=-\frac{a_{n-1}+ra_n}{(n+2)(n+1)}.\label{recurrence1}
\eea
It turns out that $\lambda(s,r)$ is an Airy function itself having the form
\bea
\lambda(s,r)=(s-r)-\frac{r}{6}(s-r)^3+\dots.\label{lambda-airy}
\eea
In the next two sections, we apply VIM to ~\eqref{Airy3n} using the first two partial sums of $\lambda(s,r)$, i.e., $\lambda(s,r)=s-r$ and $\lambda(s,r)=(s-r)-\frac{r}{6}(s-r)^3$.  Later in the following section, we prove
 that  that in both cases, VIM generates a sequence of approximate solutions which converges to the exact solution of ~\eqref{Airy3n}
\subsection{Case1:\;$\lambda=s-r$}
In this section, we take only the first term in ~\eqref{lambda-airy} and construct an iteration scheme as discussed in section ~\eqref{sec:vim}
We assume that  $\psi_n(r,t)$ have the following form
\bea
\psi_n(r,t) =e^{it}\underset{k=0}{\overset{m}{\sum}}a_k r^k,\label{psi-iteration}
\eea
where $n$ is a positive integer and represents the number of iterations that have been made from the initial guess $\psi_0(r,t)$ and $m$ is the number of terms in the $n$th iteration. 
For simplicity in the calculations, we can isolate the first two terms of the power series in ~\eqref{psi-iteration} as
\bea
\psi_n(r,t) =e^{it}\Big[a_0+a_1 r+\underset{k=2}{\overset{m}{\sum}}a_k r^k\Big].\label{psi-iterationnew}
\eea
 Substituting $\lambda(s,r)=s-r$ and ~\eqref{psi-iterationnew} into the ~\eqref{iteration-new}, Letting $V(r)=r$, and doing some algebra, we find 

\bea
\psi_{n+1}&=&e^{it}\Big[a_0+a_1 r -\frac{a_0}{2}r^2\underset{k=0}{\overset{m-2}{\sum}}a_{k+2} r^{k+2}-\underset{k=1}{\overset{m}{\sum}}\frac{a_k+a_{k-1}}{(k+2)(k+1)} r^{k+2}\nonumber \\&&
-\frac{a_m}{(m+3)(m+2)}r^{m+3}\Big]. \label{psinone}
\eea
One can observe that the first three terms and the first two sums inside the bracket in ~\eqref{psinone} have coefficients that are similar to those of the Airy function but the last term $\frac{a_m}{(m+3)(m+2)}r^{m+3}$ does not give the coefficients of the Airy function. As we see in the next section, keeping the first two terms in the series expansion for $\lambda(s,r)$ gives the same type of result. However, we will see in the next section that the sum of the terms with non-Airy coefficients accumulated in each iteration finally converges to zero as $n\to \infty$.  We skip the proof of the convergence of the solutions to the Airy function for the current case, i.e., $\lambda(s,r)=s-r$ since it is similar to the case we discuss in the next section.  \\
\subsection{Case2: $\lambda=(s-r)-\frac{r}{6}(s-r)^3$}
We next construct the iteration formula for the sequence of approximate solutions for ~\eqref{Airy3n} by letting $\lambda(s,r)$ be the second partial sum in ~\eqref{lambda-airy}, i.e.,
\bea
\lambda(s,r)=(s-r)-\frac{r}{6}(s-r)^3.\label{lambda2}
\eea
 By substitution of ~\eqref{lambda2} into ~\eqref{iteration-new} and letting $V(r)=r$, after performing the integration by-parts twice, we find
\bea
\psi_{n+1}&=&1-\frac{r^3}{2}+\int_0^r (s-r)^2\psi_n(s,t)-\int_0^r\frac{rs}{6}(s-r)^3\psi_n(s,t)ds\nonumber \\&&
-\int_0^r\lambda(s,r)\psi_{ntt}(s,r)ds.\label{it-lambda2}
\eea
We assume that $\psi_n$, the $n$-th iteration from the initial guess $\psi_0$ and satisfying ~\eqref{it-lambda2}, is of form
\bea
\psi_n
(r,t)=e^{it}\underset{k=0}{\overset{N}{\sum}}a_{k}^nr^k\label{psinguess}.
\eea
with $N\geq 1$. Substituting ~\eqref{psinguess} into ~\eqref{it-lambda2}, we find the iteration relation for $a_n$ 
\bea
a^{n+1}_{k}&=&\frac{a_{k-6}^n}{(k-1)(k-2)(k-3)(k-4)}+\frac{a_{k-5}^n}{(k-1)(k-2)(k-3)(k-4)}\nonumber \\
&&+\frac{2a_{k-3}^n}{k(k-1)(k-2)}-\frac{a_{k-2}^n}{k(k-1)}\label{AirynonAiry}
\eea
for $k\geq 6$. Begining  with $\psi_0(r,t)=e^{it}$ and by direct computation, we find $a_k^0=1$ for $k\geq 1$, $a_0^n=1$, $a_1^n=0$, $a_2^n=-\frac{1}{2}$,  $a_3^n=-\frac{1}{6}$, $a_4^n=\frac{1}{24}$, $a_5^n=\frac{1}{30}$, and $a_6^n=\frac{1}{240}$ for $n\geq 3$. Note that these first coefficients in the power series ~\eqref{psinguess} can be obtained by recursion relation in ~\eqref{Airy2-recursion-exact}. Hence, they represent the first six coefficients of the power series for the Airy coefficient. It is worth writing the coefficients of $\psi_n(r,t)$ for the first few iterations. Since we began with initial guess $\psi_0(r)=e^{it}$, therefore,
\bea
a_0^0=1, \quad a_1^0=0, \quad a_2^0=0, \quad \text{etc}.\label{coefficients1}
\eea
Given ~\eqref{coefficients1} and the iteration formula in ~\eqref{it-lambda2}, one can then obtain the non-zero coefficients of the power of $r$ for the first iterations. For $\psi_1(r,t)$, we have
\[\quad \quad \quad
a_0^1=1, \quad a_1^1=0, \quad a_2^1=-\frac{1}{2}, \quad a_3^1=-\frac{1}{6}, \quad a_4^1=0, \quad a_5^1=\frac{1}{24}, \quad a_6^1=\frac{1}{120}.\]
and for $\psi_2(r,t)$, we find
\[\quad \quad \quad
a_0^2=1,\;a_1^2=0, \;a_2^2=-\frac{1}{2}, \;a_3^2=-\frac{1}{6}, \;a_4^1=\frac{1}{24}, \;a_5^2=\frac{1}{30}, \;a_6^2=\frac{1}{180}, \; a_7^1=-\frac{1}{420},\]
\[a_8^2=-\frac{1}{1440}, \;a_9^2=-\frac{1}{15120}, \;a_{10}^1=\frac{1}{72576}, \; a_{11}^2=\frac{1}{100800}, \; a_{12}^2=\frac{1}{950400}.
\]\\
One can see in $\psi_1(r,t)$, $a_0$, $a_1$, $a_2$, and $a_3$ can be obtained by the recursion relation ~\eqref{Airy2-recursion-exact}. Hence, they give the exact terms in the series expansion of the exact solutions to ~\eqref{Airy3n}. However, the coefficeints  $a_4$, $a_5$, and $a_6$ gives some terms that are not present in the series expansion of the exact solution to ~\eqref{Airy3n}. Similarly, for $\psi_2(r,t)$, the coefficients $a_0$, $a_1$, $a_2$, $a_3$, $a_4$, and  $a_5$ are the Airy coefficients but $a_6$, $a_7$, $a_8$, $a_9$, $a_{10}$, $a_{11}$, and $a_{12}$ give some non-Airy coefficients. As we see in the next section, after each iteration, the number of terms with Airy coefficients increases by two and the terms with non-Airy coefficients increase by six.

Before we proceed to the next section, we make a remark about the rate of the convergence of $\psi_n(r,t)$ to the exact solution of ~\eqref{Airy3n}. Between the two Lagrange multipliers discussed in this section, one can expect that $\lambda=(s-r)-\frac{r}{6}(s-r)^3$ give rise to a sequence $\psi_n(r,t)$  that converges to the exact solutions of ~\eqref{Airy3n} faster compared to the case where $\lambda=s-r$, ( Fig.~\ref{fig:convergence}). \\ 
In next section, we prove that the sum of terms with non-Airy coefficients in $\psi_n$ converges to zero in the limit $n\to \infty$.
\begin{figure}[h]
\includegraphics [trim=0cm 0cm 0cm 0cm,clip=true,totalheight=0.20\textheight]{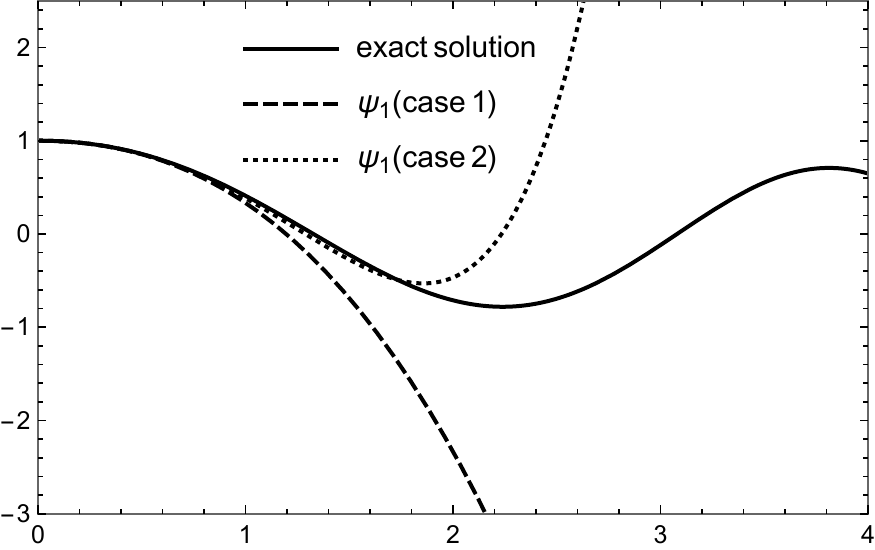}
\includegraphics [trim=0cm 0cm 0cm 0cm,clip=true,totalheight=0.20\textheight]{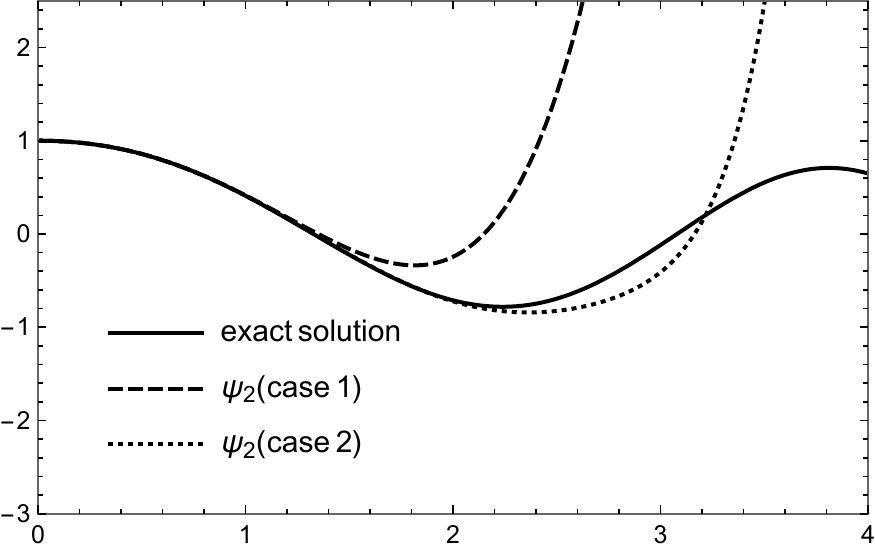}
\includegraphics [trim=0cm 0cm 0cm 0cm,clip=true,totalheight=0.20\textheight]{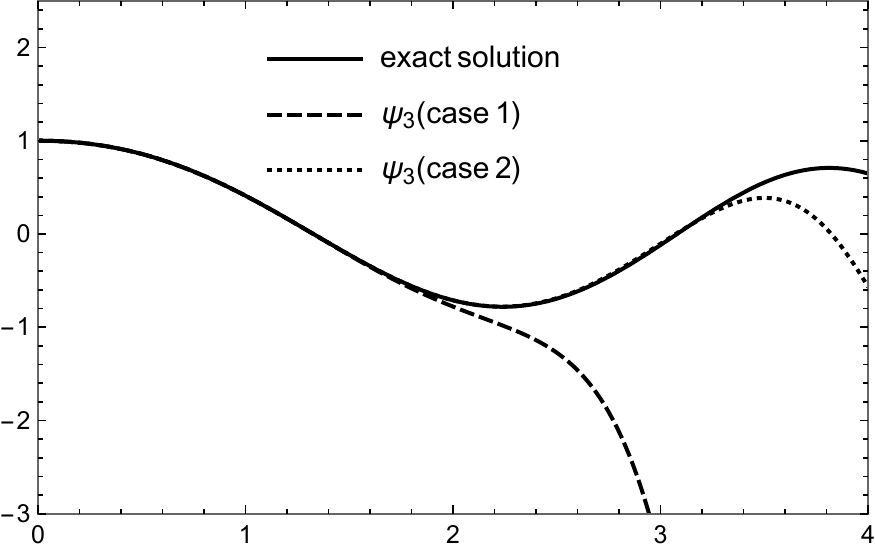}
\includegraphics [trim=0cm 0cm 0cm 0cm,clip=true,totalheight=0.20\textheight]{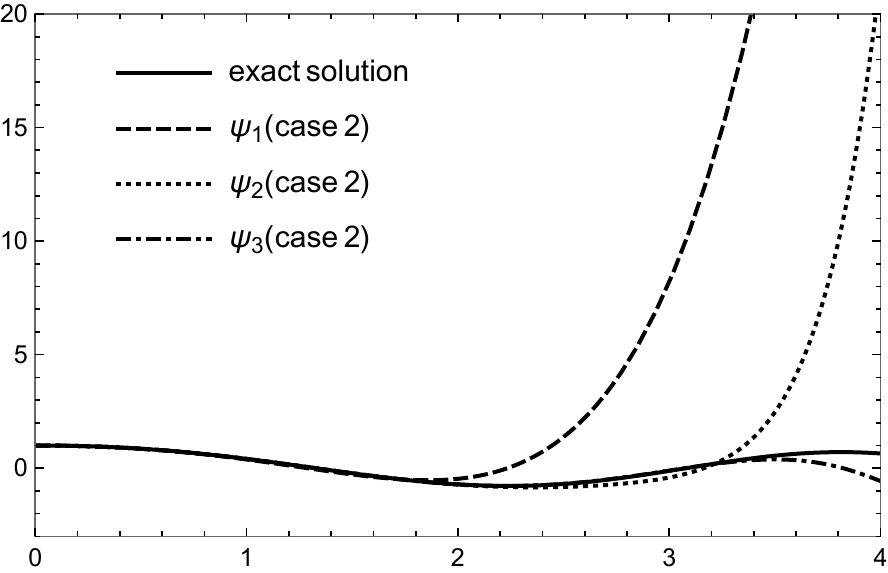}
\centering
\caption{ Figures top-left, top-right, bottom-left show the first, second and third iterations from the initial guess $\psi_0(r,t)=e^{it}$ respectively. One can see the the sequence $\psi_n$ found by the second case; i.e., $\lambda=(s-r)-\frac{r}{6}(s-r)^3$ is convergeting to the exact solution faster than the  sequence $\psi_n$ found by the first case; i.e., $\lambda=(s-r)$. The figure bottom-right depicts the exact solutions to ~\eqref{Airy3n} and the first three iterations $\psi_1$, $\psi_2$, and $\psi_3$ found in the second case.}
\label{fig:convergence}
\end{figure} 
\\
\section{Convergence of Solutions Constructed by VIM}
\label{airyconvergence}
In this section, we prove that  the sequence $\psi_n$ constructed by the iteration formula ~\eqref{it-lambda2} and the Lagrange multiplier ~\eqref{lambda2}, converges to the solutions of ~\eqref{Airy3n}. We particularly prove that the sum of the terms with non-Airy coefficients in ~\eqref{AirynonAiry} that appear in each iteration, converges to zero as the number of iterations goes to infinity.

\begin{lemma}
If  $a_{k}^n$ are the Airy coefficients for $k=0, \dots, m$, then $a_k^{n+1}$ are Airy coefficients for $k=0, \dots, m+2$.
\end{lemma}
\renewcommand\qedsymbol{$\blacksquare$}
\begin{proof}
Since $a_k^n$ is an Airy coefficients, because of ~\eqref{Airy2-recursion-exact}, we have
\bea
a_{k-3}^n=-\frac{a_{k-6}^n+a_{k-5}^n}{(k-3)(k-4)}.
\eea
Therefore,  ~\eqref{AirynonAiry} gives
\bea
a^{n+1}_{k}=-\frac{a_{k-3}^n}{(k-1)(k-2)}+
\frac{2a_{k-3}^n}{k(k-1)(k-2)}-\frac{a_{k-2}^n}{k(k-1)},\label{AirynonAiry3}
\eea
which simplifies to 
\bea
a^{n+1}_{k}=
-\frac{a_{k-3}^n+a_{k-2}^n
}{k(k-1)},\label{AirynonAiry4}
\eea
which is the desired result.
\end{proof}
\renewcommand\qedsymbol{$\blacksquare$}
\begin{lemma}
If $a_{k}^n=0$ for $k>m$, then $a_k^{n+1}=0$ for $k>m+6$. 
\end{lemma}

\begin{proof}
The desired result follows directly from ~\eqref{AirynonAiry}.
\end{proof}

\renewcommand\qedsymbol{$\blacksquare$}
\begin{lemma}
$|a_{k}^n|\leq 1$ for all $n$ and $k$.
\end{lemma}

\begin{proof}
This is clearly true for all $a_k^0$ and for all $a_k^n$ with $0\leq k\leq 6$ and all $n$. Assume $|a_k^N|\leq 1$ for all $k$. Then for $k\geq 6$, ~\eqref{AirynonAiry} gives 
\bea
|a_{k}^{N+1}|\leq \frac{2}{5!}+\frac{2}{6.5.4}+\frac{1}{6.5}\leq 1.
\eea
Hence, by induction, $|a_k^N|\leq 1$.
\end{proof}

Before proceeding to the next lemma, we shall provide a definition. Operation $\underset{\cdot}{\nmid}$ is defined as follows
\bes \bea
6\underset{\cdot}{\nmid}&=&5\cdot4\cdot3,\\
9\underset{\cdot}{\nmid}&=&8\cdot7\cdot6,\\
k\underset{\cdot}{\nmid}&=&(k-1)(k-2)(k-3)(k-6)\underset{\cdot}{\nmid}\quad \text{for} \; n\geq 4.
\eea \ees
\begin{lemma}
$|a_{k}^n|\leq \frac{2^m}{k\underset{\cdot}{\nmid}}$ for $k=6m+a$ with $a\in\{0,\;\dots\;,5\}$ and $m\in\mathbb{N}$.
\end{lemma}
\begin{proof}
For $k=6m+a$ with $a\in \{0, \dots, 5\}$\;and\; $\textit{m}\in \mathbb{N}$, using ~\eqref{AirynonAiry} we find
\bea
|a_{k}^{n+1}|\leq \frac{1}{(k-1)(k-2)}\Big(\frac{2}{(k-3)(k-4)}+1\Big)\max\Big\{|a_{k-2}^{n}|,\;\dots\;,|a_{k-6}^{n}|\Big\}.
\eea
Applying ~\eqref{AirynonAiry} to  the coefficients $\Big\{|a_{k-2}^n|,\;\dots\;,|a_{k-6}^n|\Big\}$ gives
\bea
|a_{k}^{n+1}|\leq \frac{2^2}{(k-1)(k-2)(k-7)(k-8)}\max\Big\{|a_{k-4}^{n-1}|,\;\dots\;,|a_{k-12}^{n-1}|\Big\}.
\eea
Finally, by applying  ~\eqref{AirynonAiry} to  the coefficients $\Big\{|a_{k-4}^{n-1}|,\;\dots\;,|a_{k-12}^{n-1}|\Big\}$ successively, we find
\bea
|a_{k}^{n+1}|\leq \frac{2^m}{(k-1)(k-2)(k-7)(k-8)\;\dots\;(a+1)a}\max\Big\{|a_{k-2m}^{n-m+1}|,\;\dots\;,|a_{a}^{n-m+1}|\Big\}.\label{inequality}
\eea
Then using Lemma 4.2, we have 
\bea
|a_{k}^{n+1}|\leq \frac{2^m}{k\underset{\cdot}{\nmid}}
\eea
\end{proof}
We next proceed to the main theorem of this section.
\begin{theorem}
Consider the initial value problem ~\eqref{Airy3n}. Let $\psi_0(r,t)=e^{it}$ and let $\psi_n(r,t)$ be the iteration determined by ~\eqref{it-lambda2}. Given an $R>0$, we have that $\psi_n(r,t)$ converges uniformly on $[-R,R]$ to the solution of ~\eqref{Airy3n}.
\end{theorem}
\begin{proof}
Let $R>0$ be given. By Lemma 4.4, we have that $|a_{k}^{n+1}|\leq \frac{2^m}{k\underset{\cdot}{\nmid}}$. Hence, it follows that for any given $\epsilon>0$, there exists an $N\in \mathbb{N}$ such that  $\underset{k=N+1}{\overset{\infty}{\sum}}|a_{k}^{n+1}||r^{k}| < \epsilon$ for any $n\in \mathbb{N}$ and any $|r|\leq R$. In other words, as show in Lemma 4.1, in each iteration the number of terms with Airy coefficients increases by two. As $n\to \infty$, the sum of terms with Airy coefficients converges to the solution of ~\eqref{Airy3n}. However, the sum of the terms with non-Airy coefficients converges to zero as $n\to \infty$.
\end{proof}
\renewcommand\qedsymbol{$\blacksquare$}
\section{Summary and Conclusions}
In this paper, we have formulated the VIM in a way that makes it a useful tool for studying a class of Klein-Gordon equations with variable coefficients. Particularly, we applied the VIM to a Klein-Gordon equation with a linear potential. It is assumed that the solutions to such Klein-Gordon equation can be represented by a power series. We proved that the sequence of approximate solutions produced by the iteration scheme converges uniformly to the true solutions of the Klein-Gordon equation. The sequence of iterative solutions have been studied for two particular Lagrange multipliers. It is computationally shown that the rate of the convergence to the exact solution differs for each Lagrange multiplier. Work is in progress to generalize the the result of the current work to the case where all terms in the power series of Lagrange multiplier contribute to the iteration formula and to the Klein-Gordon equation in the presence of a potential of more complicated spatial dependence and physically relevant initial data \cite{mythesis}.In this paper, we give a  formulation of  the variational iteration method that makes it suitable  for the analysis of the solutions of Klein-Gordon equations with variable coefficients.
 We particularly study a Klein-Gordon problem which has solutions in terms of Airy functions. We prove that the sequence of approximate solutions generated by the variational iteration method for such Klein-Gordon equation converges to Airy functions.

\section{Acknowledgments}


\begin{thebibliography}{0}

\bibitem{Birrell-Davies} N. D. Birrell and P. C. W. Davies, Quantum Fields in Curved Space, Cambridge University Press (1982).

\bibitem{mythesis} Gholizadeh Siahmazgi, Shohreh. On the Applications of the Variational Iteration Method to Klein-Gordon Equations. Diss. Wake Forest University, 2023.


\bibitem{boulware}  Inokuti, M., Sekine, H. and Mura, T., 1978. General use of the Lagrange multiplier in nonlinear mathematical physics. Variational method in the mechanics of solids, 33(5), pp.156-162
    

\bibitem{hej} He, J.H., 1999. Variational iteration method a kind of non-linear analytical technique: some examples. International journal of non-linear mechanics, 34(4), pp.699-708.

\bibitem{hejj} He, J.H., 2007. Variational iteration method some recent results and new interpretations. Journal of computational and applied mathematics, 207(1), pp.3-17.

\bibitem{hartle-hawking} Wazwaz, Abdul-Majid. "The variational iteration method: A reliable analytic tool for solving linear and nonlinear wave equations." Computers and Mathematics with Applications 54, no. 7-8 (2007): 926-932.


\bibitem{tatari} Tatari, M. and Dehghan, M., 2007. On the convergence of He's variational iteration method. Journal of Computational and Applied Mathematics, 207(1), pp.121-128.
    
\bibitem{kasumo}  Kasumo, C., 2020. On Exact Solutions of Klein-Gordon Equations using the Semi Analytic Iterative Method. International Journal of Advances in Applied Mathematics and Mechanics, 8(2), pp.54-63.

\bibitem{nagle} Nagle, R. K., E. B. Saff, and A. D. Snider. "Fundamentals of Differential Equations and Boundary Value Problems, 2012."







\end{thebibliography}
\end{document}